\documentclass[reqno]{amsart}

\numberwithin{equation}{section} % Equation numbering control.
\usepackage{amsmath,amsfonts,amssymb,amsthm,latexsym}
\usepackage{enumerate}
\usepackage{cancel}
\usepackage{mathrsfs}
\usepackage{stackrel}%\usepackage{ulem}

\usepackage{cases}
\usepackage[square,comma,numbers,sort&compress]{natbib}

\usepackage[colorlinks=true]{hyperref}
\hypersetup{urlcolor=blue, citecolor=red}

\newcounter{mnote}
\theoremstyle{plain}
\newtheorem{theorem}{Theorem}[section]
\newtheorem{proposition}[theorem]{Proposition}

\theoremstyle{definition}

\theoremstyle{remark}
\newtheorem{remark}[theorem]{Remark}

\newcommand{\field}[1]{\mathbb{#1}}
\newcommand{\nN}{\field{N}}

\newcommand{\nR}{\field{R}}

\newcommand{\vphi}{\varphi}
\newcommand{\veps}{\varepsilon}

\def\Om{\Omega}

\newcommand{\pd}[2]{\frac{\partial #1}{\partial #2}}
\newcommand{\od}[2]{\frac{d #1}{d #2}}

\newcommand{\abs}[1]{\left\lvert#1\right\rvert}
\newcommand{\norm}[1]{\left\lVert#1\right\rVert}
\newcommand{\Lp}[1]{\text{${L}^{#1}$}(\Om)}
\newcommand{\Hp}[1]{\text{${H}^{#1}$}(\Om)}
\begin{document}
 % ---------------------- Article Information ---------------------
\title[Data Assimilation for Benard convection in porous media]{Data Assimilation algorithm for 3D B\'enard convection in porous media employing only temperature  measurements}

\date{October 05, 2015}

% ----------------------  Author Information ----------------------
\author{Aseel Farhat}
\address[Aseel Farhat]{Department of Mathematics\\
              University of Virginia\\
       Charlottesville, VA 22904, USA}
\email[Aseel Farhat]{af7py@virginia.edu}
\author{Evelyn Lunasin}
\address[Evelyn Lunasin]{Department of Mathematics\\
                United States Naval Academy\\
        Annapolis, MD, 21401 USA}
\email[Evelyn Lunasin]{lunasin@usna.edu}

\author{Edriss S. Titi}
\address[Edriss S. Titi]{Department of Mathematics, Texas A\&M University, 3368 TAMU,
 College Station, TX 77843-3368, USA.  {\bf ALSO},
  Department of Computer Science and Applied Mathematics, Weizmann Institute
  of Science, Rehovot 76100, Israel.} \email{titi@math.tamu.edu and
  edriss.titi@weizmann.ac.il}

%%---------------------------------------------------------------
\begin{abstract}
In this paper we propose a continuous data assimilation (downscaling) algorithm for the B\'enard convection in porous media using only coarse mesh measurements of the temperature. In this algorithm, we incorporate the observables as a feedback (nudging) term in the evolution equation of the temperature. We show that under an appropriate choice of the nudging parameter and the size of the mesh, and under the assumption that the observed data is error free, the solution of the proposed algorithm converges at an exponential rate, asymptotically in time, to the unique exact unknown reference solution of the original system, associated with the observed (finite dimensional projection of) temperature data.
Moreover, we note that in the case where the observational measurements are not error free, one can estimate the error between the solution of the algorithm and the exact reference solution of the system in terms of the error in the measurements.
\end{abstract}

 \maketitle
 \noindent
 {\bf MSC Subject Classifications:} 35Q30, 37C50, 76B75, 93C20.\\
 \noindent
{\bf Keywords:} B\'enard convection, porous media, continuous data assimilation, signal synchronization, nudging, downscaling.\\

%--------------------------------------------------------------------
\section{Introduction}\label{intro}
%--------------------------------------------------------------------
Linking mathematical and computational models to a set of a coarse mesh observed data is crucial to obtain a more systematic representation of the state of a dynamical system in many physical applications. For most applications, observed data which
should ideally be defined on the whole physical domain, can be measured only discretely, often with inadequate resolution and contaminated by errors.
The problem of constructing approximate solution of dynamic systems by fusing incomplete observed data with a mathematical model has been the focus of many researchers in the past decades.   For example, designing feedback control algorithms for {\it linear time-invariant} systems has been the focus of Luenberger and many others since the early 60's and 70's (see e.g., \cite{Leunberger1971} and references therein). The problem of constructing approximate solution for {\it classes of non-linear-nth-order} systems was later studied by Thau in \cite{Thau1973}. A recent review by Nijmeijer \cite{Nijmeijer2001} addressed the problem of the synchronization from a non-linear control theory prospective.  Analyzing the validity and success of a data assimilation algorithm when some state variable observations are not available as an input on the numerical forecast model, in the context of meteorology and atmospheric physics, was also studied in \cite{Charney1969, Ghil1977, Hoke-Anthes1976, Ghil-Halem-Atlas1978}.

The classical method that was proposed, see, e.g., \cite{Charney1969} and \cite{Daley}, is
to insert observational measurements directly into a model as the latter is being integrated in time. One way to exploit this is to insert low Fourier mode observables into the equation for the evolution of the high Fourier modes. This was the approach taken for the 2D Navier-Stokes (NSE) in \cite{Browning_H_K, Henshaw_Kreiss_Ystrom, Olson_Titi_2003, Olson_Titi_2008, Korn}. The same approach was also taken by the authors in \cite{Hayden_Olson_Titi} to present a {\it discrete-in-time} data assimilation algorithm for the Lorentz and the 2D NSE systems. The authors in \cite{Hayden_Olson_Titi} and \cite{Olson_Titi_2003} were the first to study the data assimilation problem from a PDE prospective following ideas from the theory of determining modes for infinite-dimensional dynamical systems. The previously mentioned theoretical work assumed that the observational measurements are error free. We refer to some concurrent related works that allow for noisy observations in \cite{B_L_Stuart, Law-etal2014} which were motivated by the earlier analytical studies in \cite{Hayden_Olson_Titi} and \cite{Olson_Titi_2003}.

In \cite{Azouani_Titi}, the authors introduced a simple finite-dimensional feedback control algorithm for stabilizing solutions of infinite-dimensional dissipative evolution equations, such as reaction-diffusion systems, the Navier-Stokes equations and the Kuramoto-Sivashinsky equation. A computational study of this simple finite~-~dimensional feedback control algorithm was presented in \cite{Lunasin_Titi}. Based on this control algorithm, a new continuous data assimilation algorithm was developed in \cite{Azouani_Olson_Titi}.
In this new algorithm the coarse mesh observational measurements are incorporated into the equations in the form of a linear feedback control term.  That is, the  algorithm was built by adding to the model equations a term that nudges the solution towards the observations.  In the 1970s (see e.g., \cite{Hoke-Anthes1976} and references therein) this technique was called Newtonian nudging or dynamic relaxation method,  analyzed typically in much simpler scenarios. The advantage is that, since no derivatives are required
of the coarse grain observables, this works for a general class of interpolant operators.  The algorithm introduced in \cite{Azouani_Olson_Titi} was designed to work for general dissipative dynamical systems.

 The main idea
can be outlined in terms of a general dissipative evolutionary equation
\begin{align}\label{ev_eq}
\od{u}{t} = F(u),
\end{align}
where the initial data $u(0)= u^0$ is missing. The algorithm is of the form
\begin{subequations}\label{du}
\begin{align}
&\od{v}{t} = F(v) - \mu (I_h(v)- I_h(u)), \\
&v(0)= v^0,
\end{align}
\end{subequations}
where $\mu>0$ is a relaxation (nudging) parameter and $v^0$ is taken to be arbitrary initial data. $I_h(\cdot)$ represents an interpolant operator based on the observational measurements of a system at a coarse spatial resolution of size $h$, for $t\in [ 0,T ]$. Notice that if system \eqref{du} is globally well-posed and $I_h(v)$ converge to $I_h(u)$ in time, then we recover the reference $u(t,x)$ from the approximate solution $v(t,x)$. The main task is to find estimates on $\mu>0$ and $h>0$ such that the approximate solution $v(t)$ is with increasing accuracy to the reference solution $u(t)$ as more continuous data in time is supplied.

%\blue{ Analyzing the validity and success of a data assimilation algorithm when some state variable observations are not available as an input on the numerical forecast model is a multidisciplinary problem studied, in the context of meteorology and atmospheric physics (\cite{Charney1969, Hoke-Anthes1976}) and as well as in control engineering (\cite{Leunberger1971, Thau1973}) tracing back, albeit in much simpler scenarios, since the early 70's.  Related works span many decades and a more recent problem on synchronization between two systems on the basis of partial state measurements  can be found for example in \cite{Nijmeijer2001}.
%}

It is worth noting that the continuous data assimilation in the context of the incompressible 2D NSE studied in \cite{Azouani_Olson_Titi} under the assumption that the data is noise free. The case when the observational data contains stochastic noise is treated in \cite{Bessaih-Olson-Titi}.
The authors in \cite{Bessaih-Olson-Titi} established resolution conditions which guarantee that the
limit supremum, as the time tends to infinity, of the expected value
of the $L^2$-norm of the difference between the approximating solution
and the actual reference solution, (i.e. the error), is bounded by an estimate involving the variance
of the noise in the measurements and the spatial resolution of the collected data, $h$.

%The analysis of uncertainty quantification  for the B\'enard convection problem in porous medium, in the case where the measurements and observations contain noise,  is a subject of future work by a subset of us.

It is also worthwhile to mention that the authors in \cite{MTT2015} analyzed an algorithm for continuous data assimilation for 3D Brinkman-Forchheimer-extended Darcy (3D BFeD) model of a porous medium, a model equation when the velocity is too large for classical Darcy's law to be valid. A similar algorithm for stochastically noisy data is at hand combining ideas from the work in \cite{Bessaih-Olson-Titi} and \cite{FLT2015}. Furthermore, the proposed data assimilation algorithm can also be applied to several three-dimensional subgrid scale turbulence models. In \cite{ALT2014}, it was shown that approximate solutions constructed using observations on all three components of the unfiltered velocity field converge in time to the reference solution of the 3D NS-$\alpha$ model. An abridged data assimilation algorithm for the 3D Leray-$\alpha$ model, using observations in {\it only any two components and without any measurements on the third component of the velocity field}, is analyzed in \cite{FLT2015_alpha}.

A particular application that we have in mind in this paper is flows in porous media, subjected to heating from below and cooling from the top. Flows in porous media are connected to many important problems in geophysics (see, e.g., \cite{Straughan} and references therein) as well as in biological systems and biotechnology (see, e.g., \cite{Vafai}).

We consider $\Omega = [0,L]\times [0,l]\times [0,1]$ to be a box in $\nR^3$ of porous media saturated with a fluid. The side walls of the box are insulated, and the box is heated from below with constant temperature $T_0$ and cooled from above with constant temperature $T_1<T_0$. After some change of variables and proper scaling, (see, e.g. \cite{Doering, Straughan}), the governing non-dimensional equations of the convecting fluid through a porous medium are: 
\begin{subequations}\label{3D_Benard_Porous}
\begin{align}
\gamma \pd{u}{t} + u + \nabla p &= Ra\,\theta \hat{k}, \label{3D_Benard_Porous_u}\\
\pd{\theta}{t} - \Delta \theta + (u\cdot\nabla)\theta -u\cdot \hat{k}&= 0, \label{3D_Benard_Porous_theta}\\
\nabla\cdot u&=0, \label{div_free}\\
u(0;x,y,z) = u^0(x,y,z), \quad \theta(0;x,y,z) &= \theta^0(x,y,z),
\end{align}
subject to the boundary conditions:
\begin{align}
\theta(t; x,y,0) = \theta(t; x,y,1) & =0, \label{bd_2}\\
\pd{\theta}{x}(t; 0,y,z) = \pd{\theta}{x}(t; L,y,z) = \pd{\theta}{y}(t; x,0,z) = \pd{\theta}{y}(t; x,l,z) &=0, \label{bd_3}
\end{align}
and
\begin{align}
u\cdot\hat{n} &=0, \quad \text{on }\partial\Omega. \label{bd_1}
\end{align}
\end{subequations}
Here, $u(t; x, y, z)=(u_1(t;x,y,z), u_2(t;x,y,z), u_3(t;x,y,z))$ is the fluid velocity, $p=p(t;x,y,z)$ is the pressure. $\theta=\theta(t;x,y,z)$ is the scaled fluctuation of the temperature around the steady state background temperature profile $(T_1-T_0)z+T_0$  and it is given by $\theta = T- (\frac{T_0}{T_0-T_1}-z)$, where $T=T(t;x,y,z)$ is the temperature of the fluid inside the box $\Omega$, $T_0$ is the temperature of the fluid at the bottom and $T_1$ is the temperature of the fluid at the top. The non-dimensional parameter $\gamma$ is equal to $\frac{1}{Pr}$, where $Pr$ is the Darcy-Prandtl number representing a measure of the ratio of the viscosity to the thermal diffusion coefficient. The non-dimensional parameter $Ra$ is the Rayleigh-Darcy number, which is a measure of the ratio of the driving force (coming from the imposed average temperature gradient), to the viscosity and thermal diffusion in the system. The vector $\hat{k}$ is the unit vector in the $z$-direction and $\hat{n}=\hat{n}(x,y,z)$ is the normal vector to the boundary $\partial \Omega$ at the point $(x,y,z)$. We remark that the temperature fluctuation $\theta(t;x,y,z)$ satisfies an advection-diffusion equation \eqref{3D_Benard_Porous_theta}. Also, the evolution of the fluid velocity $u(t;x,y,z)$ and pressure $p(t;x,y,,z)$ is described by Darcy's law \eqref{3D_Benard_Porous_u}, this replaces the conservation of momentum equations in the Navier-Stokes equations.

The B\'enard convection in porous media problem \eqref{3D_Benard_Porous} with $\gamma >0$ (namely, the Darcy-Prandtl number $Pr$ is finite), was studied in \cite{Fabrie_1} and \cite{Fabrie_2}. There, the existence and uniqueness of global weak and strong solutions were established.   It was also shown that the temperature $\theta(t; x,y,z)$ satisfies the maximum principle and has an absorbing ball in $L^p(\Omega)$ for each finite $p$. The authors of \cite{Fabrie_Nicolaenko} showed that for initial data $(u^0, \theta^0) \in L^2(\Omega)\times L^\infty(\Omega)$, the system has a finite-dimensional global attractor $\mathcal{A}\subset L^2(\Omega)\times L^\infty(\Omega)$ which attracts in the $L^2(\Omega)\times L^2(\Omega)$ metric. Moreover, they showed that the system has exponential attractors. The Gevrey regularity (spatial analyticity) of the solutions on the global attractor was later studied in \cite{Oliver_Titi}. There, a rigorous lower bound estimate on the radius of analyticity was obtained.

When $\gamma =0$ (namely, the Darcy-Prandtl number $Pr$ is infinite), system \eqref{3D_Benard_Porous} was studied in \cite{Ly_Titi}. All the constraints on the boundary are similar to the nonzero Darcy-Prandtl number case, while we note that in the $\gamma =0$ case, there is no need to specify the initial data $u^0(x,y,z)$.  One can recover $u^0(x,y,z)$ by solving equation \eqref{3D_Benard_Porous_u} given $\theta^0(x,y,z)$. This is unlike the non-zero Darcy-Prandtl case where one has to specify both initial data $u^0(x,y,z)$ and $\theta^0(x,y,z)$. It was shown in \cite{Ly_Titi} that the system for the $\gamma =0$ case has global real analytic solutions and admits a real analytic global attractor. Moreover, it was also shown that the standard Galerkin solution of the system, when $\gamma=0$, converges exponentially fast, in the wave number, to the exact solution of the system.
This in turn justifies the computational results reported in \cite{G_S_T}.

We propose a  continuous data assimilation algorithm for the 3D B\'enard convection in porous media employing coarse mesh measurements of only the temperature field. We incorporate the observables as a feedback term in the original evolution equation for the system. This algorithm can be implemented with a variety of finitely many observables: for instance low Fourier modes, nodal values, finite volume averages, or finite elements. Our algorithm is given by the following system:
\begin{subequations}\label{DA_3D_Benard_Porous}
\begin{align}
\gamma \pd{v}{t} + v + \nabla q &= Ra\,\eta \hat{k}, \label{DA_3D_Benard_Porous_v}\\
\pd{\eta}{t} - \Delta \eta + (v\cdot\nabla)\eta - v\cdot \hat{k}  &= -\mu(I_h(\eta)-I_h(\theta)), \label{DA_3D_Benard_Porous_eta}\\
\nabla\cdot v&=0, \label{DA_div_free}\\
v(0;x,y,z) = v^0(x,y,z), \quad \eta(0;x,y,z) &= \eta^0(x,y,z),
\end{align}
subject to the boundary conditions:
\begin{align}
\eta(t; x,y,0) = \eta(t; x,y,1) & =0, \label{DA_bd_2}\\
\pd{\eta}{x}(t; 0,y,z) = \pd{\eta}{x}(t; L,y,z) = \pd{\eta}{y}(t; x,0,z) = \pd{\eta}{y}(t; x,l,z) &=0, \label{DA_bd_3}
\end{align}
and
\begin{align}
v\cdot\hat{n} &=0, \quad \text{on } \partial\Omega.\label{DA_bd_1}
\end{align}
\end{subequations}
Here, $q$ is a modified pressure, and we note that the initial conditions $v_0$ and $\eta_0$, may chosen arbitrarily, (for e.g., both initial conditions are set to zero). In this paper, we will only consider interpolant observables that are given by a linear interpolant operator $I_h: \Hp{1} \rightarrow \Lp{2}$ satisfying
the approximation property
\begin{align}\label{app}
\norm{\varphi - I_h(\varphi)}_{\Lp{2}} \leq c_0h\norm{\varphi}_{\Hp{1}},
\end{align}
for every $\varphi \in \Hp{1}$, where $c_0>0$ is a dimensionless constant.  One example of an interpolant observable that satisfies \eqref{app} is the orthogonal projection onto the low Fourier modes with wave numbers less than $1/h$. A more physical example are the volume elements that were studied in \cite{Jones_Titi}. The algorithm and our analysis in this paper will carry on to the second type of interpolants treated in \cite{Azouani_Olson_Titi}. This second type is given by a linear operator $I_h: \Hp{2}\rightarrow\Lp{2}$, together with
\begin{align}\label{app2}
\norm{\varphi - I_h(\varphi)}_{\Lp{2}} \leq c_1h\norm{\varphi}_{\Hp{1}} + c_2h^2\norm{\varphi}_{\Hp{2}},
\end{align}
for every $\varphi \in \Hp{2}$, where $c_1, c_2>0$ are dimensionless constants. An example of this type of interpolant observables that satisfies \eqref{app2} is given by the measurements at a discrete set of nodal points in $\Omega$. We are not treating the second type of interpolants in this paper only for the simplicity of presentation. For full details on the analysis for the second type of interpolants we refer to \cite{FJT} and \cite{FLT2015}.

%Our current work was inspired to sidetrack a common problem in data assimilation where the dimension of the observation vector is less than the dimension of the model's state vector.

Our current analytical investigation is motivated by  some recent studies in \cite{FJT} and \cite{FLT2015}. In \cite{FJT},  a continuous data assimilation scheme for the two-dimensional incompressible B\'enard convection problem was introduced. The data assimilation algorithm in \cite{FJT} constructs the approximate solutions for the velocity $u$ and temperature fluctuations $\theta$ using only the observational data, $I_h(u)$, of the velocity field and without any measurements for the temperature (or density) fluctuations. Inspired by the recent algorithms proposed in \cite{Azouani_Olson_Titi, FJT}, we introduced an abridged dynamic continuous data assimilation for the 2D Navier-Stokes equations (NSE) in \cite{FLT2015}. The proposed algorithm in \cite{Azouani_Olson_Titi} for the 2D NSE requires measurements for the two components of the velocity vector field. On the other hand, in \cite{FLT2015}, we establish convergence results for an improved algorithm where the observational data needed to be measured and inserted into the model equation is reduced or subsampled. Our algorithm there requires observational measurements of only one component of the velocity vector field.
Ideally, one would like to design an algorithm, for the B\'enard convection model, based on temperature measurements only, but the analysis to support or dispute the existence of such an algorithm is a subject of ongoing research. In fact, the authors of this paper believe that knowing the temperature trajectory does not, in general, determine the velocity vector field in the 2D B\'enard convection model (see \cite{Altaf} for a computational study supporting this claim). It is worthwhile to mention that earlier conjecture of Charney {\it et. al.} \cite{Charney1969}, that complete temperature history of the atmosphere for certain simple atmospheric models, will determine other state variables, was tested in \cite{Ghil1977, Ghil-Halem-Atlas1978} (see also \cite{Altaf}). It was noted in \cite{Ghil-Halem-Atlas1978} that in practice, the complexities in weather forecast models together with numerous problems in the real data make it uncertain that assimilation of temperature data alone will yield initial states of arbitrary accuracy. The authors in \cite{Ghil-Halem-Atlas1978} developed and tested different time-continuous data assimilation methods using temperature data. Their results suggested that temperature data can have an impact on the numerical weather forecasts that is sensitive to the quality of such observed data and is highly affected by the assimilation method used. The authors concluded that designing more refined methods for data assimilation and developing better numerical models can lead a major improvement in numerical weather prediction using temperature measurements alone.

Our main contribution in the current study is to give a rigorous justification that a data assimilation algorithm based on temperature measurements alone can be designed for the 3D B\'enard convection in porous media. We provide explicit estimates on the relaxation (nudging) parameter $\mu$ and the spatial resolution $h$ of the observational measurements, in terms of physical parameters, that are needed in order for the proposed downscaling algorithm to recover the reference solution under the assumption that the supplied data are error free. In the case where the observational measurements are not error free, one can estimate the error between the solution of the algorithm and the exact reference solution of the system in terms of the error in the measurements.
While the typical scenario in data assimilation is to choose $\mu$ depending on $h$, in our convergence analysis we choose our parameters $\mu$ and $h$ to depend on physical parameters. More explicitly, we choose $\mu$ to depend on the bounds of the solution on the global attractor of the system and then choose $h$ to depend on $\mu$. The philosophy here is that in order to prove the convergence theorems, we need to have a complete resolution of the flow, so $h$ has to depend on the physical parameters, such as the Rayleigh number ($Ra$) in this case.

Numerical simulations in \cite{Gesho} and \cite{Gesho_Olson_Titi} (see also \cite{Hayden_Olson_Titi}) have shown that, in the absence
of measurements errors, the continuous data assimilation algorithm for the 2D Navier-Stokes equations performs much better than analytical estimates in \cite{Azouani_Olson_Titi} would suggest. This was also noted in a different context in \cite{Olson_Titi_2003} and \cite{Olson_Titi_2008}. It is likely that the data assimilation algorithm studied in this paper will also perform much better than suggested by the analytical results, i.e. under more relaxed conditions than those assumed in the rigorous estimates. This is a subject of future work.

The outline of the paper is as follows. In section \ref{pre}  we give some preliminaries. In section \ref{zero}, we present the convergence analysis of our proposed data assimilation algorithm in the case of zero Darcy-Prandtl case and the nonzero Darcy-Prandtl case in section \ref{nonzero}.

\bigskip

%---------------------------------------------------------------------
\section{Preliminaries}\label{pre}
%---------------------------------------------------------------------

For the sake of completeness, this section presents some preliminary material and notation commonly used in the mathematical study of hydrodynamics models, in particular in the study of the Navier-Stokes equations (NSE) and the Euler equations. For more detailed discussion on these topics, we refer the reader to, e.g., \cite{Constantin_Foias_1988}, \cite{Robinson}, \cite{Temam_1995_Fun_Anal} and \cite{Temam_2001_Th_Num}.

Let $\Lp{p}:=L^p:\Omega\rightarrow \nR$ and $\Hp{k}:= H^k:\Omega\rightarrow \nR$ be denote the usual $L^p$ Lebesgue space and $H^k$-Sobolev space, respectively, for $1\leq p\leq \infty$ and $k\in\nR$. We define the spaces:
\begin{align*}
\mathcal{V} &:= \{ u \in (C^\infty(\Omega))^3: \, u\cdot\hat{n} =0 \text{ on } \partial\Omega \text{ and } \nabla\cdot u =0 \text{ in } \Omega\},\\
\tilde{\mathcal{V}} &:= \{\theta \in C^\infty(\Omega): \, \theta \text{ satisfies the boundary conditions } \eqref{bd_2} \text{ and }\eqref{bd_3}\},\\
{\bf H} &:= \text{the closure of } \mathcal{V} \text{ in the } (\Lp{2})^3 \text{ norm}, \\
{\bf V} &:= \text{the closure of } \mathcal{V} \text{ in the } (\Hp{1})^3 \text{ norm}, \\
H &:= \text{the closure of } \tilde{\mathcal{V}} \text{ in the } \Lp{2} \text{ norm}, \\
V &:= \text{the closure of } \tilde{\mathcal{V}} \text{ in the } \Hp{1} \text{ norm}.
\end{align*}

We define the inner products on ${\bf H}$ and $H$ by
\[(u,w)=\sum_{i=1}^3\int_{\Om} u_iw_i\,dxdydz, \quad \text{and} \quad (\theta,\eta) = \int_{\Om} \theta\eta\,dxdydz,
\]
respectively.
We also define the inner product on ${\bf V}$ and $V$ by
\[((u,w))=\sum_{i,j=1}^3\int_{\Om}\partial_ju_i\partial_jw_i\,dxdydz, \quad \text{and} \quad ((\theta,\eta))=\sum_{j=1}^3\int_{\Om}\partial_j\theta\partial_j\eta\,dxdydz,
\]
respectively.
Note that, thanks to the boundary conditions \eqref{bd_2} and \eqref{bd_3}, $\norm{\cdot}_{V}=((\cdot,\cdot))^{1/2} = \norm{\nabla \cdot}_{\Lp{2}}$ is a norm on $V$ due to the Poincar\'e inequality \eqref{poincare_1}, below.

We define the Helmholtz-Leray projector $P_\sigma$ as the orthogonal projection from $(\Lp{2})^3$ onto ${\bf H}$ and define $A = -\Delta$ subject to the boundary condition \eqref{bd_1} and \eqref{bd_3} with the domain
$$\mathcal{D}(A)= \{\theta\in \Hp{2}: \theta \text{ satisfies }\eqref{bd_2} \text{ and }\eqref{bd_3}\}.$$ Using the Lax-Milgram Theorem and the elliptic regularity in the box $\Omega$, the linear operator $A$ is self-adjoint and positive definite with compact inverse $A^{-1}:~ H~\rightarrow~ \mathcal{D}(A)$. Thus, there exists a complete orthonormal set of eigenfunctions $w_i$ in $H$ such that $Aw_i= \lambda_iw_i$ where $0<\lambda_i\leq\lambda_{i+1}$ for $i\in \nN$.

We define the bilinear from $\mathcal{B}(\cdot,\cdot): {\bf V}\times \mathcal{D}(A)\rightarrow H$, such that
$$\mathcal{B}(u, \theta) := (u\cdot\nabla) \theta, $$
for every $u\in {\bf V}$ and $\theta \in \mathcal{D}(A)$. Using the boundary conditions \eqref{bd_2} and \eqref{bd_3}, one can easily check that
\begin{align}\label{zero_nonlinearity}
(\mathcal{B}(u, \theta), \theta) =0,
\end{align}
for every $u\in {\bf V}$ and $\theta \in \mathcal{D}(A)$.

We recall the Poincar\'e inequality:
\begin{align}\label{poincare_1}
\|\vphi\|_{\Lp{2}}^2\leq \lambda_1^{-1}\|A^{1/2}\vphi\|_{\Lp{2}}^2, \quad \text{ for all } \vphi \in V,
\end{align}
where $\lambda_1$ is the smallest eigenvalue of the operator $A$.

Let $Y$ be a Banach space. We denote by $L^p([0,T];Y)$ the space of (Bochner) measurable functions $t\mapsto w(t)$, where $w(t)\in Y$, for a.e. $t\in[0,T]$, such that the integral $\int_0^T\|w(t)\|_Y^p\,dt$ is finite.

Furthermore, inequality \eqref{app} implies that
\begin{align}\label{app_F}
 \norm{\theta-I_h(\theta)}_{\Lp{2}}\leq c_0h\norm{A^{1/2}\theta}_{\Lp{2}},
\end{align}
for every $\theta\in V$.

We recall the following existence and uniqueness results for the 3D B\'enard convection problem in porous media \eqref{3D_Benard_Porous}. Hereafter, $c$ will denote a universal dimensionless positive constant that may change from line to line.
\begin{proposition}\cite{Ly_Titi}
Given $\theta \in V$, there exists a unique solution $u\in {\bf V}$ of the problem \eqref{3D_Benard_Porous_u} and \eqref{div_free}, with $\gamma =0$, subject to the boundary condition \eqref{bd_1}. Moreover, $u$ satisfies
\begin{align}\label{u_to_theta_0}
\norm{u}_{{\bf V}} \leq c Ra \norm{\theta}_{V}.
\end{align}
\end{proposition}

\begin{theorem}\cite{Ly_Titi}\label{unif_bounds_0}
Let $\theta^0\in V$, then system \eqref{3D_Benard_Porous}, with $\gamma=0$, has a unique global strong solution $(u,\theta)$ that satisfies:
\begin{align*}
u \in C([0,T];{\bf V})\cap L^2([0,T]; (\Hp{2})^3),\quad \text{and}\quad \od{u}{t} \in L^2([0,T];{\bf H}),\\
\theta \in C([0,T];V)\cap L^2([0,T];\mathcal{D}(A)),\quad \text{and}\quad \od{\theta}{t} \in L^2([0,T];H),
\end{align*}
for any $T>0$, and depends continuously on the initial data $\theta^0$. Moreover, the solution satisfies the following bound:
\begin{align}
\limsup_{t\rightarrow \infty} \norm{\theta(t)}_{\Lp{\infty}} \leq 1, \label{L_infty_0}
\end{align}
and the system admits a compact finite-dimensional global attractor $\mathcal{A}\subset{\bf V}\times V$.
\end{theorem}

\begin{theorem}\label{gamma_existence}\cite{Fabrie_Nicolaenko}
If $u^0 \in {\bf H}$ and $\theta^0\in \Lp{\infty}$, then for any $T>0$, system \eqref{3D_Benard_Porous}, with $\gamma >0$, has a unique weak solution $(u,\theta)$ that satisfies
\begin{align*}
u \in L^\infty([0,T]; {\bf H}),\quad\text{and} \quad
\theta \in L^\infty([0,T]; \Lp{\infty})\cap L^2([0,T]; V).
\end{align*}
%If $\theta^0\in M_b$ for some $b\geq 1$, then the solution $\theta(t) \in M_b$ a.e. $t\geq 0$.
Moreover, for every finite $p$,
\begin{align}
\limsup_{t\rightarrow \infty} \norm{\theta(t)}_{\Lp{p}} \leq \abs{\Omega}^{1/p}. \label{L_p_gamma}
\end{align}
More precisely, $\theta(t;x,y,z)$ can be decomposed such that
\begin{align}\label{L_infty_gamma}
\theta(t;x,y,z)=\theta_1(t;x,y,z) + \theta_2(t;x,y,z),
\end{align}
where $-1\leq \theta_1(t;x,y,z) \leq 1$,  and that there exist two positive constants $\alpha_1$ and $\alpha_2$ depending on $p$ and $\Omega$ such that
$$\norm{\theta_2(t)}_{\Lp{p}} \leq \alpha_1 e^{-\alpha_2 t}.$$
If the initial data $\theta^0$ satisfies $m\leq \theta^0(x,y,z)\leq M$, $a.e.$ in $\Omega$, for some $m\geq -1$ and $M\leq 1$, we have
\begin{align}
m \leq \theta(t;x,y,z) \leq M, \quad a.e. \text{ in } \Omega,
\end{align}
for all $t\geq 0$.

Finally, if $u^0 \in {\bf V}$ and $\theta^0\in \Lp{\infty}$, then there exists a unique strong solution $(u,\theta)$ that satisfies
\begin{align*}
u \in C([0,T]; {\bf V}),\quad \text{and}\quad \od{u}{t} \in L^2([0,T];{\bf H}),\\
\theta \in L^\infty([0,T]; \Lp{\infty})\cap L^2([0,T]; \mathcal{D}(A)),\quad \text{and}\quad \od{\theta}{t} \in L^2([0,T];H),
\end{align*}
and the system admits a compact finite-dimensional exponential attractor $\mathcal{A}\subset{\bf V}\times V$.
\end{theorem}

\bigskip
%----------------------------------------------------------------------
\section{Convergence Results}\label{conv}
%------------------------------------------------------------------------
In this section, we will establish the global existence, uniqueness and stability of solutions of system \eqref{DA_3D_Benard_Porous}, when the observable data satisfy \eqref{app}.

%\begin{theorem}[Existence and Uniqueness]
%Suppose $I_h$ satisfy \eqref{app} with $\mu>0$ and $h>0$ are chosen such that $\mu c_0^2h^2\leq 1$, where $c_0$ is the constant in \eqref{app}.
%\begin{enumerate}
%\item For $\gamma=0$, if the initial data $\eta^0 \in V$, then for any $T>0$, the continuous data assimilation system \eqref{DA_3D_Benard_Porous}, subject to the boundary conditions \eqref{DA_bd_1}--\eqref{DA_bd_3}, possess a unique global strong solution $(v,\eta)$ that satisfies
%\begin{align*}
%v \in C([0,T];{\bf V})\cap L^2([0,T]; (\Hp{2})^3),\quad \text{and}\quad \od{v}{t} \in L^2([0,T];{\bf H}),\\
%\eta \in C([0,T];V)\cap L^2([0,T];\mathcal{D}(A)),\quad \text{and}\quad \od{\eta}{t} \in L^2([0,T];H).
%\end{align*}
%\item For $\gamma>0$, if the initial data $v^0\in {\bf V}$ and $\eta^0 \in V$, then for any $T>0$, the continuous data assimilation system \eqref{DA_3D_Benard_Porous}, subject to the boundary conditions \eqref{DA_bd_1}--\eqref{DA_bd_3}, possess a unique global strong solution $(v,\eta)$ that satisfies
%\begin{align*}
%v \in C([0,T]; {\bf V}),\quad \text{and}\quad \od{v}{t} \in L^2([0,T];{\bf H}) \\
%\eta \in C([0,T];V)\cap L^2([0,T];\mathcal{D}(A)),\quad \text{and}\quad \od{\eta}{t} \in L^2([0,T];H).
%\end{align*}
%\end{enumerate}
%\end{theorem}

Since we assume that $(u,\theta)$ is a reference solution of system \eqref{3D_Benard_Porous}, then it is enough to show the existence and uniqueness of the difference $(w,\xi)=(u-v, \theta-\eta)$. In the proofs of Theorem \ref{th_conv_0} and Theorem \ref{th_conv_gamma} below, we will drive formal \textit{a-priori} bounds on the difference $(w, \xi)$, under the conditions that $\mu$ is large enough and $h$ is small enough such that $\mu c_0^2h \leq 1$. These \textit{a-priori} estimates, together with the global existence and uniqueness of the solution $(u,\theta)$, form the key elements for showing the global existence of the solution $(v,\eta)$ of system \eqref{DA_3D_Benard_Porous}. The convergence of the approximate solution $(v,\eta)$ to the exact reference solution $(u,\theta)$ will also be established under the same  conditions on the nudging parameter $\mu$ stated in \eqref{mu_to_Ra_0} and \eqref{mu_to_Ra_gamma}, when $\gamma=0$ and $\gamma>0$, respectively. Uniqueness can then be obtained using similar energy estimates.

The estimates we provide in this section are formal, but can be justified by the Galerkin approximation procedure
and then passing to the limit while using the relevant compactness theorems. We
will omit the rigorous details of this standard procedure (see, e.g., \cite{Constantin_Foias_1988, Robinson, Temam_2001_Th_Num, Fabrie_1, Fabrie_Nicolaenko, Ly_Titi}) and provide only the formal \textit{a-priori} estimates.

%------------------------------------------------------------------------
\subsection{The infinite Darcy-Prandtl number case ($\gamma=0$)} \label{zero}
%------------------------------------------------------------------------

\begin{theorem}\label{th_conv_0}
Let $I_h$ satisfy the approximation property \eqref{app} and $(u,\theta)$ be a strong solution in the global attractor of \eqref{3D_Benard_Porous} with $\gamma=0$. Suppose that $v^0\in {\bf H}$ and $\eta^0\in H$. Assume  that $\mu>0$ is large enough such that
\begin{align}\label{mu_to_Ra_0}
\mu +\frac{\lambda_1}{2} \geq 2cRa^2 +4Ra,
\end{align}
 and that $h>0$ is small enough such that $\mu c_0^2h^2\leq 1$. Then, there exists a unique weak solution $(v,\eta)$ of system \eqref{DA_3D_Benard_Porous} with $\gamma =0$ such that
\begin{align*}
v \in C([0,T];{\bf H})\cap L^2([0,T]; ({\bf V}),\quad \text{and}\quad \od{v}{t} \in L^2([0,T]; {\bf V}^{'}),\\
\eta \in C([0,T];H)\cap L^2([0,T];V),\quad \text{and}\quad \od{\eta}{t} \in L^2([0,T];V^{'}).
\end{align*}

Moreover, $\norm{u(t)-v(t)}_{\Lp{2}}^2 \rightarrow 0$, and $\norm{\theta(t)-\eta(t)}_{\Lp{2}}^2$ $\rightarrow 0$, at an
exponential rate, as $t \rightarrow \infty$.
\end{theorem}

\begin{proof}
Define $w:=u-v$ and $\xi:=\theta-\eta$. Then, in functional settings, $w$ and $\xi$ will satisfy the equations:
\begin{subequations}\label{difference_0}
\begin{align}
w &= RaP_\sigma(\xi \hat{k}), \label{DA_3D_Benard_Porous_w_0}\\
\pd{\xi}{t} +  A\xi + \mathcal{B}(v,\xi) + \mathcal{B}(w, \theta) -w\cdot \hat{k} &=-\mu I_h(\xi).\label{DA_3D_Benard_Porous_xi_0}
%\nabla\cdot w&=0. \label{DA_div_free_0}
\end{align}
\end{subequations}
Taking the $H$-inner product of \eqref{DA_3D_Benard_Porous_xi_0} with $\xi$, we obtain
\begin{align*}
\frac{1}{2} \od{}{t}\norm{\xi}_{\Lp{2}}^2 +  \norm{A^{1/2} \xi}_{\Lp{2}}^2 &+(\mathcal{B}(v,\xi),\xi) + (\mathcal{B}(w,\theta),\xi) \notag \\
& = (w\cdot \hat{k}, \xi) - \mu (I_h(\xi), \xi).
\end{align*}

Thanks to \eqref{zero_nonlinearity}, we have
\begin{align}\label{1_0}
(\mathcal{B}(v,\xi),\xi) = 0.
\end{align}
We also notice from \eqref{DA_3D_Benard_Porous_w_0} that
\begin{align*}
w= Ra P_\sigma(\xi\hat{k}), \quad \text { in } L^2([0,T]; {\bf H}),
\end{align*}
for any $T>0$. This implies that
\begin{align}\label{2_0}
(w\cdot \hat{k},\xi) = Ra(P_\sigma(\xi\hat{k}), \xi) \leq Ra\norm{P_\sigma(\xi\hat{k})}_{\Lp{2}}\norm{\xi}_{\Lp{2}}\leq Ra\norm{\xi}_{\Lp{2}}^2.
\end{align}

Using and H\"older's inequality, we get
%\begin{align*}
%\abs{(\mathcal{B}(w,\theta),\xi)}& \leq \norm{w}_{\Lp{6}} \norm{\xi}_{\Lp{3}} \norm{A^{1/2}\theta}_{\Lp{2}}\notag \\
%& \leq c \norm{w}_{\bf V}\norm{\xi}_{\Lp{2}}^{1/2}\norm{A^{1/2} \xi}_{\Lp{2}}^{1/2}\norm{A^{1/2} \theta}_{\Lp{2}}.
%\end{align*}
%Thanks to \eqref{u_to_theta_0} and Young�s inequality we have
%\begin{align}\label{3_0}
%\abs{(\mathcal{B}(w,\theta),\xi)}&\leq c Ra\norm{\xi}_{\Lp{2}}^{1/2}\norm{A^{1/2}\xi}_{\Lp{2}}^{3/2}\norm{A^{1/2}\theta}_{\Lp{2}}\notag \\
%& \leq \frac{1}{4} \norm{A^{1/2}\xi}_{\Lp{2}}^2 + cRa^{4} \norm{A^{1/2}\theta}_{\Lp{2}}^4 \norm{\xi}_{\Lp{2}}^2.
%\end{align}
\begin{align*}
\abs{(\mathcal{B}(w,\theta),\xi)} & = \abs{(\mathcal{B}(w,\xi),\theta)}\notag\\ 
& \leq \norm{w}_{\Lp{2}} \norm{\theta}_{\Lp{\infty}} \norm{A^{1/2}\xi}_{\Lp{2}}\notag \\
& \leq \frac{1}{8} \norm{A^{1/2}\xi}_{\Lp{2}}^2 + c\norm{\theta}_{\Lp{\infty}}^2 \norm{w}_{\Lp{2}}^2.
\end{align*}
Thanks to the equation \eqref{DA_3D_Benard_Porous_w_0} we have
\begin{align}\label{3_0}
\abs{(\mathcal{B}(w,\theta),\xi)} &\leq \frac{1}{8} \norm{A^{1/2}\xi}_{\Lp{2}}^2 + cRa^2\norm{\theta}_{\Lp{\infty}}^2 \norm{\xi}_{\Lp{2}}^2.
\end{align}
The approximation inequality \eqref{app_F} and Young's inequality imply
\begin{align}\label{4_0}
-\mu (I_h(\xi),\xi) &= -\mu (I_h(\xi)-\xi, \xi) - \mu \norm{\xi}_{\Lp{2}}^2 \notag\\
&\leq \mu \norm{I_h(\xi)-\xi}_{\Lp{2}}\norm{\xi}_{\Lp{2}} - \mu \norm{\xi}_{\Lp{2}}^2 \notag \\
& \leq \frac{\mu c_0^2h^2}{2} \norm{A^{1/2}\xi}_{\Lp{2}}^2 - \frac{\mu}{2}\norm{\xi}_{\Lp{2}}^2.
\end{align}

Using assumption $\mu c_0^2h^2 \leq 1$ and estimates \eqref{1_0}--\eqref{4_0}, we conclude that
\begin{align*}
\od{}{t} \norm{\xi}_{\Lp{2}}^2 + \frac{1}{2}\norm{A^{1/2}\xi}_{\Lp{2}}^2 \leq \left(cRa^2\norm{\theta}_{\Lp{\infty}}^2 +2 Ra -\mu\right) \norm{\xi}_{\Lp{2}}^2.
\end{align*}
The Poincar\'e inequality \eqref{poincare_1} yields
\begin{align*}
\od{}{t} \norm{\xi}_{\Lp{2}}^2 + \frac{\lambda_1}{2}\norm{\xi}_{\Lp{2}}^2 \leq \left(cRa^2\norm{\theta}_{\Lp{\infty}}^2 +2 Ra -\mu\right) \norm{\xi}_{\Lp{2}}^2.
\end{align*}
We define $$\alpha(t):= \mu +\frac{\lambda_1}{2}- 2 Ra - cRa^2\norm{\theta}_{\Lp{\infty}}^2.$$
Then,
\begin{align}
\od{}{t}\norm{\xi}_{\Lp{2}}^2 + \alpha(t)\norm{\xi}_{\Lp{2}}^2 \leq 0.
\end{align}

The uniform bound \eqref{L_infty_0} implies that: for any fixed $\veps>0$, there exists a time $t_0(\veps)>0$ such that
\begin{align*}
\norm{\theta(t)}_{\Lp{\infty}}\leq 1+\veps,
\end{align*}
for all $t\geq t_0$. Taking $\veps =1$, then there exists a time $t_0>0$ such that $\norm{\theta(t)}_{\Lp{\infty}} \leq 2$, for all $t\geq t_0$. Then, we have
\begin{align*}
\alpha(t) & \geq \mu + \frac{\lambda_1}{2}- 2 Ra - cRa^2,
\end{align*}
and
\begin{align*}
\alpha(t)& \leq \mu + \frac{\lambda_1}{2}+ 2 Ra + cRa^2<\infty,
\end{align*}
for all $t\geq t_0$. Now, the assumption \eqref{mu_to_Ra_0} implies that
\begin{align}
\alpha(t)\geq  2 Ra +cRa^2 >0,
\end{align}
for all $t\geq t_0$.

By Gronwall's Lemma, it follows that
\begin{align*}
\norm{\theta(t)-\eta(t)}_{\Lp{2}}^2  = \norm{\xi(t)}_{\Lp{2}}^2 \rightarrow 0,
\end{align*}
at an exponential rate, as $t\rightarrow \infty$. Finally, the equation \eqref{DA_3D_Benard_Porous_w_0} yields that
\begin{align*}
\norm{w(t)}_{\Lp{2}}^2 \leq Ra \norm{\xi(t)}_{\Lp{2}}^2,
\end{align*}
thus,
\begin{align*}
\norm{u(t)-v(t)}_{\Lp{2}}^2= \norm{w}_{\Lp{2}}^2 \rightarrow 0,
\end{align*}
at an exponential rate, as $t\rightarrow \infty$.

\end{proof}

\bigskip
%------------------------------------------------------------------------
\subsection{The finite Darcy-Prandtl number case ($\gamma>0$)} \label{nonzero}
%------------------------------------------------------------------------
\begin{theorem}\label{th_conv_gamma}

Fix $\gamma>0$, and $(u,\theta)$ be a strong solution in the global attractor of \eqref{3D_Benard_Porous}. Suppose that $v^0\in {\bf H}$, $\eta^0\in H$ and  that $I_h$ satisfies the approximation property \eqref{app}. Let $\mu>0$ be large enough such that
\begin{align}\label{mu_to_Ra_gamma}
2 \mu + \lambda_1 \geq  2\frac{cRa^4}{\gamma} +2c\gamma(1+\lambda_1^{-1})^2,
\end{align}
and $h>0$ is small enough such that $\mu c_0^2h^2\leq 1 $. Then, there exists a unique weak solution $(v,\eta)$ of system \eqref{DA_3D_Benard_Porous} corresponding to the same $\gamma>0$ such that
\begin{align*}
v \in C([0,T]; {\bf H}),\quad \text{and}\quad \od{v}{t} \in L^2([0,T];{\bf V}^{'}) \\
\eta \in C([0,T];H)\cap L^2([0,T];V),\quad \text{and}\quad \od{\eta}{t} \in L^2([0,T];V^{'}).
\end{align*}

Moreover, $\norm{u(t)-v(t)}_{\Lp{2}}^4 + \norm{\theta(t)-\eta(t)}_{\Lp{2}}^4$ $\rightarrow 0$, at an exponential rate, as $t \rightarrow \infty$.
\end{theorem}
\begin{proof}
Define $w:=u-v$ and $\xi:=\theta-\eta$. Then, in functional settings, $w$ and $\xi$ will satisfy the equations:
\begin{subequations}\label{difference_gamma}
\begin{align}
\gamma \od{w}{t} + w &= Ra P_\sigma(\xi \hat{k}), \label{DA_3D_Benard_Porous_w_gamma}\\
\od{\xi}{t} + A\xi + \mathcal{B}(v,\xi) + \mathcal{B}(w, \theta)- w\cdot \hat{k} &=-\mu I_h(\xi). \label{DA_3D_Benard_Porous_xi_gamma}
%\nabla\cdot w&=0. \label{DA_div_free_gamma}
\end{align}
\end{subequations}

Taking the ${\bf H}$-inner product of \eqref{DA_3D_Benard_Porous_w_gamma} with $w$ and the $H$-inner product of\eqref{DA_3D_Benard_Porous_xi_gamma} with $\xi$, respectively, we obtain
\begin{align*}
\frac{\gamma}{2}\od{}{t}\norm{w}_{\Lp{2}}^2 + \norm{w}_{\Lp{2}}^2 &= Ra (\xi, w\cdot\hat{k}),\notag\\
\frac{1}{2} \od{}{t}\norm{\xi}_{\Lp{2}}^2 + \norm{A^{1/2} \xi}_{\Lp{2}}^2 &+(\mathcal{B}(v,\xi),\xi) + (\mathcal{B}(w,\theta),\xi) \notag \\
& = (w\cdot \hat{k}, \xi) - \mu (I_h(\xi), \xi).
\end{align*}
Thanks to \eqref{zero_nonlinearity}, we have
\begin{align}\label{1_gamma}
(\mathcal{B}(v,\xi),\xi) = 0.
\end{align}
Also, estimate \eqref{4_0} shows that
\begin{align}\label{2_gamma}
-\mu (I_h(\xi),\xi)
& \leq \frac{\mu c_0^2h^2}{2} \norm{A^{1/2}\xi}_{\Lp{2}}^2 - \frac{\mu}{2}\norm{\xi}_{\Lp{2}}^2.
\end{align}
By H\"older's inequality, we have
\begin{align}\label{3_gamma}
\abs{(\mathcal{B}(w,\theta),\xi)} & = \abs{(\mathcal{B}(w,\xi),\theta)}\notag\\ 
& \leq \norm{w}_{\Lp{2}} \norm{\theta}_{\Lp{\infty}} \norm{A^{1/2}\xi}_{\Lp{2}}\notag \\
& \leq \frac{1}{8} \norm{A^{1/2}\xi}_{\Lp{2}}^2 + c\norm{\theta}_{\Lp{\infty}}^2 \norm{w}_{\Lp{2}}^2.
\end{align}

The above estimates \eqref{1_gamma}--\eqref{3_gamma}, the Cauchy-Schwarz inequality, and the Poincar\'e inequality \eqref{poincare_1} yield
\begin{align*}
\frac{\gamma}{2}\od{}{t}\norm{w}_{\Lp{2}}^2 + \norm{w}_{\Lp{2}}^2 &= Ra (\xi, w\cdot\hat{k})\notag\\
& \leq \frac{1}{2} \norm{w}_{\Lp{2}}^2 + \frac{Ra^2}{2} \norm{\xi}_{\Lp{2}}^2\notag \\
\frac{1}{2} \od{}{t}\norm{\xi}_{\Lp{2}}^2 + \norm{A^{1/2} \xi}_{\Lp{2}}^2 +\frac\mu2 \norm{\xi}_{\Lp{2}}^2 &\leq \frac{\mu c_0^2h^2}{2}\norm{A^{1/2}\xi}_{\Lp{2}}^2 + \frac{1}{4}\norm{A^{1/2}\xi}_{\Lp{2}}^2\notag \\ &\quad +c \norm{\theta}_{\Lp{\infty}}^2\norm{w}_{\Lp{2}}^2  + c\lambda_1^{-1}\norm{w}_{\Lp{2}}^2.
\end{align*}
Using the assumption $\mu c_0^2 h^2 \leq1$ imply
\begin{subequations}
\begin{align}
\od{}{t}\norm{w}_{\Lp{2}}^2 + \frac{1}{\gamma} \norm{w}_{\Lp{2}}^2 &\leq \frac{Ra^2}{\gamma} \norm{\xi}_{\Lp{2}}^2, \label{L2_w}\\
\od{}{t}\norm{\xi}_{\Lp{2}}^2 + \left(\mu+\frac{\lambda_1}{2}\right)\norm{\xi}_{\Lp{2}}^2 &\leq c\left(\lambda_1^{-1}+\norm{\theta}_{\Lp{\infty}}^2\right)\norm{w}_{\Lp{2}}^2. \label{L2_xi}
\end{align}
\end{subequations}

Multiplying \eqref{L2_w} with $\norm{w}_{\Lp{2}}^2$ and \eqref{L2_xi} with $\norm{\xi}_{\Lp{2}}^2$, respectively, we have
\begin{align*}
\frac{1}{2} \od{}{t} \norm{w}_{\Lp{2}}^4 + \frac{1}{\gamma} \norm{w}_{\Lp{2}}^4 &\leq \frac{Ra^2}{\gamma} \norm{w}_{\Lp{2}}^2 \norm{\xi}_{\Lp{2}}^2,\notag\\
\frac{1}{2} \od{}{t} \norm{\xi}_{\Lp{2}}^4 + \left(\mu+\frac{\lambda_1}{2}\right)\norm{\xi}_{\Lp{2}}^4 &\leq c\left(\lambda_1^{-1}+\norm{\theta}_{\Lp{\infty}}^2\right)\norm{w}_{\Lp{2}}^2\norm{\xi}_{\Lp{2}}^2. 
\end{align*}
Using the Cauchy-Schwarz inequality, we get
\begin{align*}
&\frac{1}{2} \od{}{t}\left(\norm{w}_{\Lp{2}}^4 + \norm{\xi}_{\Lp{2}}^4\right) \notag \\
&\qquad \qquad  \leq -\frac{1}{\gamma} \norm{w}_{\Lp{2}}^4 - \left(\mu+\frac{\lambda_1}{2}\right)\norm{\xi}_{\Lp{2}}^4 + \frac{Ra^2}{\gamma} \norm{w}_{\Lp{2}}^2\norm{\xi}_{\Lp{2}}^2 \notag \\ &\qquad \qquad \quad+  c\left(\lambda_1^{-1}+\norm{\theta}_{\Lp{\infty}}^2\right)\norm{w}_{\Lp{2}}^2\norm{\xi}_{\Lp{2}}^2 \notag \\
&\qquad \qquad \leq -\frac{1}{\gamma}\norm{w}_{\Lp{2}}^4 - \left(\mu+\frac{\lambda_1}{2}\right)\norm{\xi}_{\Lp{2}}^4 + \frac{1}{4\gamma}\norm{w}_{\Lp{2}}^4 + \frac{cRa^4}{\gamma}\norm{\xi}_{\Lp{2}}^4 \notag \\ & \qquad \qquad \quad+ \frac{1}{4\gamma}\norm{w}_{\Lp{2}}^4 + c \gamma\left(\lambda_1^{-1}+\norm{\theta}_{\Lp{\infty}}^2\right)^2\norm{\xi}_{\Lp{2}}^4\notag\\
& \qquad \qquad= -\frac{1}{2\gamma} \norm{w}_{\Lp{2}}^4 - \left(\mu +\frac{\lambda_1}{2}- \frac{cRa^4}{\gamma} - c \gamma\left(\lambda_1^{-1}+\norm{\theta}_{\Lp{\infty}}^2\right)^2\right)\norm{\xi}_{\Lp{2}}^4. 
\end{align*}
The above inequality can be rewritten as
\begin{align}
\od{}{t}\left(\norm{w}_{\Lp{2}}^4 + \norm{\xi}_{\Lp{2}}^4\right) + \alpha(t)\left(\norm{w}_{\Lp{2}}^4 + \norm{\xi}_{\Lp{2}}^4\right)\leq 0,
\end{align}
where $$\alpha(t) = \min\bigg\{\frac{1}{\gamma},\, 2\mu + \lambda_1-\frac{cRa^4}{\gamma}-c \gamma\left(\lambda_1^{-1}+\norm{\theta(t)}_{\Lp{\infty}}^2\right)^2\bigg\}.$$

Recall that, by Theorem \eqref{gamma_existence}, the solution $\theta(t;x,y,z)$ of the 3D B\'enard problem in porous media satisfies the Maximum Principle and 
$$\limsup_{t\rightarrow \infty} \norm{\theta(t)}_{\Lp{\infty}} \leq 1. $$
Then, by a similar argument as in the proof of the previous theorem, we conclude that there exists a time $t_0>0$ such that
\begin{align*}
\alpha(t)& \geq  \min\bigg\{\frac{1}{\gamma}, \, 2\mu+\lambda_1 -\frac{cRa^4}{\gamma}-c\gamma(1+\lambda_1^{-1})^2\bigg\},
\end{align*}
and
\begin{align*}
\alpha(t) \leq \min\bigg\{\frac{1}{\gamma}, \, 2\mu + \lambda_1+ \frac{cRa^4}{\gamma} +c\gamma(1+\lambda_1^{-1})^2\bigg\}<\infty,
\end{align*}
for all $t\geq t_0$. Now,  assumption \eqref{mu_to_Ra_gamma} implies that
\begin{align}
\alpha(t) \geq \min\bigg\{\frac{1}{\gamma}, \, \frac{cRa^4}{\gamma}+c\gamma(1+\lambda_1^{-1})^2\bigg\} >0,
\end{align}
for all $t\geq t_0$.

By Gronwall's Lemma , it follows that
\begin{align*}
\norm{u(t)-v(t)}_{\Lp{2}}^4 + \norm{\theta(t)-\eta(t)}_{\Lp{2}}^4 \rightarrow 0,
\end{align*}
at an exponential rate, as $t\rightarrow \infty$.
\end{proof}

\begin{remark}
 Suppose that initial data $(v^0,\eta^0)$ is contained in double the size of the absorbing ball (for example, one $(v^0,\eta^0)=(0,0)$). Consider the   approximate solutions of system \eqref{DA_3D_Benard_Porous} with interpolant operators $I_h$ satisfying either of the approximate properties \eqref{app} or \eqref{app2}. Then one can find  $\mu>0$ large enough and $h>0$  small enough, depending on the physical parameters of the underlying system, such that $v(t)$ converges in the strong sense to the reference solutions of system \eqref{3D_Benard_Porous}, i.e. $\norm{u(t)-v(t)}_{\bf V} + \norm{\theta(t)-\eta(t)}_{V} \rightarrow 0$, at an exponential rate, as $t\rightarrow \infty$.  The proof of this statement follows similar adjusted arguments and ideas to those presented  in \cite{FJT} and \cite{FLT2015}.
\end{remark}
\bigskip
%-----------------------------------------------------------------------

\section*{Acknowledgements}
E.S.T. is thankful to the kind hospitality of the \'Ecole Polytechnique (CMLS), Paris, where part of this work was completed. The work of A.F. is supported in part by NSF grant  DMS-1418911. The work of E.L. is supported  by the ONR grant N0001415WX01725. The work of  E.S.T.  is supported in part by the ONR grant N00014-15-1-2333 and the NSF grants DMS-1109640 and DMS-1109645.
\bigskip
%---------------------------------------------------------------

%---------------------------------------------------------------------
\end{document}